\documentclass{article}
\usepackage[T1]{fontenc}
\usepackage{times} 
\usepackage{textcomp}
\usepackage{amsmath,amsthm,amssymb} 
\usepackage{mathptmx}
\usepackage{color}
\usepackage{soul}
\usepackage{enumitem}
\usepackage{booktabs}
\usepackage{graphicx}

\newtheorem{proposition}{Proposition}[section]
\newtheorem{theorem}[proposition]{Theorem}

\newtheorem{lemma}[proposition]{Lemma}

\newtheorem{assumption}[proposition]{Assumption}
\newtheorem{remark}[proposition]{Remark}

\newcommand{\mathbbm}{\mathbb} 
\newcommand{\mathscr}{\mathcal}

\DeclareMathOperator*{\argmax}{arg\,max}

\newcommand{\Rb}{\mathbb{R}}

\newcommand{\1}{\mathbbm{1}}

\newcommand{\Fb}{\mathbb{F}}

\newcommand{\Eb}{\mathbb{E}}

\newcommand{\Hc}{\mathcal{H}}

\newcommand{\Lc}{\mathcal{L}}
\newcommand{\Nc}{\mathcal{N}}

\def\half{\textstyle{\frac{1}{2}}}

\newenvironment{tightitemize}{%
    \list{{\textup{$\bullet$}}}{\settowidth\labelwidth{{\textup{\qquad}}}
    \leftmargin\labelwidth \advance\leftmargin\labelsep
    \parsep 0pt plus 1pt minus 1pt \topsep 3pt \itemsep 3pt
    }}{\endlist}

\newenvironment{tightlist}[1]{%
    \list{{\textup{(\roman{enumi})}}}{\settowidth\labelwidth{{\textup{(#1)}}}
    \leftmargin 6pt \advance\leftmargin\labelsep \itemindent \parindent
    \parsep 0pt plus 1pt minus 1pt \topsep 0pt \itemsep 0pt
    \usecounter{enumi}}}{\endlist}

\begin{document}
\title{A Dual Method For Evaluation of Dynamic Risk\\ in Diffusion Processes}
\author{
 Andrzej Ruszczy\'nski\footnote{
Rutgers University, Department of Management Science and Information Systems, Piscataway, NJ 08854, USA, Email: rusz@rutgers.edu}
\and Jianing Yao\footnote{RBC Capital Markets, New York, NY 10281 Email: {yaojn{\_}1@hotmail.com}}}

%
%
\date{}
%

%
%

%
\maketitle

\begin{abstract}
We propose a numerical method for risk evaluation defined by a backward stochastic differential equation. Using dual representation of the risk measure, we convert the risk evaluation to a simple stochastic control problem where the control is a certain Radon-Nikodym derivative process. By exploring the maximum principle, we show that a piecewise-constant dual control provides a good approximation on a short interval. A  dynamic programming algorithm extends the approximation to a finite time horizon. Finally, we illustrate the application of the procedure to financial risk management in conjunction with nested simulation and on a multidimensional
portfolio valuation problem.\\

\end{abstract}


\section*{Introduction}

The main objective of this paper is to present a simple and efficient numerical method for solving backward stochastic differential equations with convex and homogeneous drivers. Such equations are fundamental modeling tools for continuous-time dynamic risk measures with Brownian filtration, but may also arise in other applications.

The key property of dynamic risk measures is \emph{time-consistency}, which allows for dynamic programming formulations. The discrete time case was extensively explored by Detlfsen and Scandolo \cite{DS1}, Bion-Nadal \cite{BN}, Cheridito \emph{et al.} \cite{CD1,CK1}, F\"{o}llmer and Penner \cite{FP1},
Fritelli and Rosazza Gianin \cite{FR1},
Frittelli and Scandolo \cite{FS},  Riedel \cite{RF}, and Ruszczy\'{n}ski and Shapiro \cite{AR2}.

For the continuous-time case, Coquet,  Hu,  M{\'e}min and Peng \cite{coquet2002filtration} discovered that time-consistent dynamic risk measures, with Brownian filtration, can be represented as solutions of \emph{Backward Stochastic Differential Equations} (BSDE) \cite{PP1}; under mild growth conditions, this is the only form possible. Specifically,
the $y$-part solution of one-dimensional BSDE, defined below, measures the  risk of a variable $\xi_T$ at the current time $t$:
\begin{align}\label{eq:bsde1}
Y_t = \xi_T + \int_t^T g(s, Z_s)\;ds - \int_t^T Z_s\;dW_s, \quad 0\leq t \leq T,
\end{align}
with the driver $g$ being interpreted as a ``risk rate.'' The $\mathcal{F}_T$-measurable random variable $\xi_T$ is usually a function of the terminal state of a certain stochastic dynamical system.

Inspired by that, Barrieu and El Karoui provided a comprehensive study in \cite{BE1,BE2};  further contributions being made by Delbean, Peng, and Rosazza Gianin \cite{delbaen2010representation}, and Quenez and Sulem \cite{quenez2013bsdes} (for a more general model with Levy processes). In addition, application to finance was considered, for example, in \cite{laeven2014robust}. Using the convergence results of Briand, Delyon and M\'emin \cite{BDM}, Stadje \cite{stadje} finds the drivers of BSDE corresponding to discrete-time risk measures.

Motivated by an earlier work on risk-averse control of discrete-time stochastic process \cite{AR4}, Ruszczy\'nski and Yao \cite{AY} formulate a risk-averse stochastic control problem for diffusion processes. The corresponding dynamic programming
equation leads to a decoupled forward--backward system of stochastic differential equations.

While forward stochastic differential equations can be solved by several efficient methods, the main challenge is the numerical solution of \eqref{eq:bsde1}, where $\xi_T$ represents the future value function.
 In particular, Zhang \cite{zjf} and Bouchard and Touzi \cite{touzi} use backward Euler's approximation and regression. Such an approach, however, is not well-suited for risk measurement, because it does not
  preserve the monotonicity of the risk measure.  Alternatively,  {\O}ksendal and Sulem \cite{Oksendal} directly attack continuous-time risk-averse control problem with jumps by deriving sufficient conditions. Algorithms based on maximum principle were investigated by  Ludwig \emph{et al.} \cite{ludwig}. A class of methods using Monte-Carlo approximations were developed by Gobet \emph{et al.} in \cite{gobet1,gobet2}. Using branching processes was explored by Henry-Labordere, Tan, and Touzi \cite{Henry}.

  Our idea is to derive a recursive method based on duality relations in risk-averse dynamic programming, so that the approximation becomes a time-consistent coherent risk measure in discrete time. We believe that exploring such ideas may become advantageous within other algorithmic schemes as well.

The paper is organized as follows. In section \ref{s:problem}, we quickly introduce the concept of a dynamic risk measure and review its properties.  In section \ref{s:dual-control}, we recall the dual representation of a dynamic risk measure and formulate an equivalent stochastic control problem. The optimality condition for the dual control problem, a special form of a maximum principle, are derived in section \ref{s:maximum-principle}. Section \ref{s:estimates} estimates the errors introduced by using
constant processes as dual controls. In section \ref{s:discrete-approximation}, we present the whole numerical method with piecewise-constant dual controls and analyze its rate of convergence. Finally, in section \ref{s:risk-management}, we illustrate the efficacy of our approach on a two-stage risk management model and a multi-dimensional
risk valuation problem for a portfolio.

\section{The Risk Evaluation Problem}
\label{s:problem}
Given a complete filtered probability space $(\varOmega, \mathcal{F}, \mathbb{P})$ with filtration $\{\mathcal{F}_t\}_{t\in [0,T]}$ generated by $d$-dimensional Brownian motion $\{W_t\}_{t\in [0,T]}$, we consider
the following stochastic differential equation:
\begin{equation}
\label{SDE}
dX_t = b(t,X_t)\;dt + \sigma(t,X_t)\; dW_t,\,\,X_0 = x,\quad t\in [0,T],
\end{equation}
with measurable $b:[0,T]\times\Rb^n\to \Rb^n$, and $\sigma:[0,T]\times\Rb^n\to \Rb^n\times\Rb^d$.
We introduce the following notation.
\begin{tightitemize}
\item $\mathbb{E}_{t}[\,\cdot\,] : = \mathbb{E}[\,\cdot\,|\,\mathcal{F}_t]$;
\item $\Lc^2(\varOmega, \mathcal{F}_t, \mathbb{P};\Rb^n)$: the set of $\Rb^n$-valued  $\mathcal{F}_t$-measurable random variables $\xi$ such that $\| \xi \|^2 := \mathbb{E}[\,|\xi|^2\,]<\infty$; for $n=1$, we write it $\Lc^2(\varOmega, \mathcal{F},\mathbb{P})$;
\item $\Hc^{2, n}[t, T]$: the set of $\Rb^n$-valued adapted processes $Y$ on $[t, T]$, such that
$
\| Y \|^2_{\Hc^{2,n}[t, T]} := \mathbb{E}_t\Big[ \int_t^T |Y_s|^2\,d s\Big] < \infty$;
 for $n = 1$ we write it $\Hc^{2}[t, T]$;\footnote{When the norm is clear from the context, the subscripts are skipped.}
\end{tightitemize}

Our intention is to evaluate risk of a terminal cost generated by the forward process \eqref{SDE}:
\begin{equation}
\label{rho}
\rho_{0,T}\big[ \varPhi(X_T)\big],
\end{equation}
where $\varPhi:\Rb^n\to\Rb$, and $\big\{\rho_{s,t}\big\}_{0\le s \le t\le T}$ is a
dynamic risk measure consistent with the filtration $\{\mathcal{F}_t\}_{t\in [0,T]}$. We refer the reader to \cite{PSG1} for a comprehensive discussion on risk measurement and filtration-consistent evaluations.

Special role in the dynamic risk theory is played by
\emph{$g$-evaluations} which are defined by one-dimensional backward stochastic differential equations
of the following form:
\begin{equation}
\label{BSDE}
dY_t = -g(t,Y_t, Z_t)\;dt + Z_t\;dW_t,\quad Y_{T} = {\varPhi}(X_{T}),\,\,t\in [0,T],
\end{equation}
with $\rho^g_{t,T}\big[ \varPhi(X_T)]$ defined to be equal to $Y_t$.

We make the following general assumptions about the functions appearing in the forward--backward system:

\vspace{0.5ex}
\begin{assumption}
 \label{a:c01b}
 {\quad} \vspace{-2.5ex}\\
 \begin{tightlist}{ii}
 \item { The functions $b, \sigma,  \varPhi, g$ are deterministic and uniformly Lipschitz continuous with respect to $x$, $y$, and $z$;}
 \item { The functions $b, \sigma,g$ are uniformly  $\frac{1}{2}$-H\"{o}lder continuous with respect to $t$;}
 \item The dimensions of the Brownian motion and the state process coincide, i.e., $n = d$, and $c>0$ exists, such that
\[
\sigma(t, x)\,\sigma^{\!\top}(t, x) \geq {c}\,\mathbb{I}, \quad \forall (t, x) \in [0, T]\times\mathbb{R}^d.
\]
 \end{tightlist}
\end{assumption}

As proved in \cite{coquet2002filtration} every $\Fb$-consistent nonlinear evaluation, which is dominated by a $g$-evaluation with $g = \mu|y| + \nu|z|$ with some
$\nu, \mu > 0$, is in fact a $g$-evaluation for some $g$; the dominance is understood as follows:
\[
\rho_{0, T}[\xi + \eta] - \rho_{0, T}[\xi] \leq \rho_{0, T}^{\nu, \mu}[\eta]
\]
for all $\xi, \eta\in \Lc^2(\varOmega, \mathcal{F}_T, \mathbb{P})$.

\begin{proposition}
\label{p:nonev}
For all $0\leq t\leq T$ and all $\xi, \xi' \in \Lc_2(\varOmega,\mathcal{F}_T, \mathbb{P})$, the following properties hold:
\begin{tightlist}{iii}
\item \textbf{Generalized constant preservation}: If $\xi\in \Lc_2(\varOmega,\mathcal{F}_t, \mathbb{P})$, then $\rho^g_{t,t}[\,\xi\,]=\xi$;
\item \textbf{Time consistency}: $\rho^g_{s,T}[\,\xi\,] = \rho^g_{s,t}[\,\rho^g_{t,T}[\,\xi\,]\,]$, for all  $0\le s\leq t$;
\item \textbf{Local property}:
 $\rho^g_{t,T}[\,\xi\1_A+\xi'\1_{A^c}\,]= \1_A \rho^g_{t,T}[\,\xi\,]+\1_{A^c} \rho^g_{t,T}[\,\xi'\,]$, for  all $A\in \mathcal{F}_t$.
\end{tightlist}
\end{proposition}

From now on, we focus exclusively on $g$-evaluations as
dynamic risk measures, and we  skip the superscript $g$ in $\rho^g[\cdot]$.

The evaluation of risk is equivalent to the solution of a \emph{decoupled} forward--backward system of stochastic differential equations \eqref{SDE}--\eqref{BSDE}.
An important virtue of this system is its \emph{Markov property}:
\begin{equation}
\label{rho-Markov}
\rho_{t,T}\big[\varPhi(X_T)\big]= v(t,X_t),
\end{equation}
where $v:[0,T]\times\Rb^n\to \Rb$. We have
\begin{equation}
\label{risk-value-function}
v(t,x) = \rho^x_{t,T}\big[ \varPhi(X_T^{t,x})\big],\quad (t,x)\in [0,T]\times\Rb^n,
\end{equation}
where $\{X_s^{t,x}\}$ is the solution of the system \eqref{SDE} restarted at time $t$ from state $x$:
\begin{equation}
\label{SDEtx}
dX_s^{t,x} = b(s,X_s^{t,x})\;ds + \sigma(s,X_s^{t,x})\; dW_s,\quad s\in [t,T],\quad X_t^{t,x} = x,
\end{equation}
and $\rho^x_{t,T}\big[ \varPhi(X_T^{t,x})\big]$ is the (deterministic) value of $Y_t^{t,x}$ in the backward
equation \eqref{BSDE} with terminal condition $\varPhi(X_T^{t,x})$.

Numerical methods for solving forward equations are very well understood (see, e.g., \cite{NSDE}). We focus, therefore, on the backward equation \eqref{BSDE}. So far, a limited number of results are available for this purpose. The most prominent is the Euler method with functional regression  (see, e.g., \cite{mjdouglas, mj, touzi, zjfthesis,zjf,zjf-book}). Our intention is to show that for drivers satisfying additional coherence conditions, a much more effective method can be developed, which exploits time-consistency, duality theory for risk measures, and the maximum principle in stochastic control.

\section{The Dual Control Problem}
\label{s:dual-control}
We further restrict the risk measures under consideration to coherent measures, by making the following additional assumption about the driver~$g$.

\vspace{1ex}
\begin{assumption}
\label{a:riskmeasure}
The driver $g$ satisfies for almost all $t\in [0,T]$ the following conditions:\vspace{0.5ex}
\begin{tightlist}{iii}
\item $g$ is independent of $y$, that is, $g:[0,T]\times \Rb^d\to\Rb$;
\item $g(t, \cdot)$ is convex for all $t\in [0,T]$;
\item $g(t,\cdot)$ is positively homogeneous for all $t\in [0,T]$.
\end{tightlist}
\end{assumption}
\vspace{1ex}

Under these conditions, one can derive further properties of the evaluations $\rho_{t,T}^g[\cdot]$ for $t\in [0,T]$, in addition to the general properties of $\Fb$-consistent nonlinear expectations stated in Proposition \ref{p:nonev}.

\vspace{1ex}
\begin{theorem}
\label{t:geval-prop}
Suppose $g$ satisfies Assumption \ref{a:riskmeasure}. Then the dynamic risk measure $\{\rho_{t,r}\}_{0\le t \le r \le T}$ has the following properties:
\begin{tightlist}{iv}
\item \textbf{Translation Property}: for all $\xi\in \Lc_2(\varOmega, \mathcal{F}_r, \mathbb{P})$ and $\eta\in \Lc_2(\varOmega, \mathcal{F}_t, \mathbb{P})$,
\[
\rho_{t,r}[\xi+\eta] = \rho_{t,r}[\xi]+\eta, \quad \text{a.s.};
\]
\item \textbf{Convexity}: for all $\xi,\xi'\in \Lc_2(\varOmega, \mathcal{F}_r, \mathbb{P})$
and  all $\lambda \in L^\infty(\varOmega, \mathcal{F}_t, \mathbb{P})$ such that \mbox{$0 \leq \lambda \leq 1$},
\begin{equation*}
\rho_{t,r}[\lambda \xi + (1-\lambda)\xi']\leq \lambda \rho_{t,r}[\xi]+(1-\lambda)\rho_{t,r}[\xi'],\quad \text{a.s.}.
\end{equation*}
\item \textbf{Positive Homogeneity}: for all $\xi\in \Lc_2(\varOmega, \mathcal{F}_r, \mathbb{P})$
and all $\gamma \in L^\infty(\varOmega, \mathcal{F}_t, \mathbb{P})$ such that $\gamma \ge 0$,
we have
\[
\rho_{t,r}[\gamma \xi] = \gamma \rho_{t,r}[\xi], \quad \text{a.s.}.
\]
\end{tightlist}
\end{theorem}
\vspace{1ex}

It follows that under Assumption  \ref{a:riskmeasure}, the operators $\{\rho_{t,r}\}_{0\leq t\leq r\leq T}$ constitute a family of \emph{coherent} conditional measures of risks.
Particularly important for us is the  dual representation of the risk measure $\big\{\rho_{t,T}\big\}_{t\in[0,T]}$, which is based on the dual representation of the driver $g$:
\begin{equation}
\label{dual-driver}
g(t,z) = \max_{{a}\in A_t}  {a}  z,
\end{equation}
where $A_t = \partial_z g(t,0)$ is a bounded, convex, and closed set in $\Rb^n$.
Here, and elsewhere, ${a} z$ is understood as the scalar product of ${a}$ and $z$.
The following statement simplifies the results of Barrieu and El Karoui \cite[sec. 7.3]{BE2} to our case.

\vspace{1ex}
\begin{theorem}\label{thm:coherentdual}
Suppose Assumption \ref{a:riskmeasure} is satisfied. Then
\begin{align}
\label{eq:dual}
\rho_{t, T}^g\big[\varPhi(X_T^{t,x})\big]
= \sup_{\mu\in\mathcal{A}_{t,T}}\mathbb{E}\big[\,\varGamma_{t,T} \varPhi(X_T^{t,x})\,\big],
\end{align}
where $\mathcal{A}_{t,T}$ is the space of $A_s$-valued  adapted processes on $[t,T]$, and the process $\big\{\varGamma_{t,s}\big\}_{s\in[t,T]}$ in $\Hc^2[t,T]$ satisfies the stochastic differential equation:
\begin{align}\label{eq:densityprocess}
d\varGamma_{t,s} = \mu_s \varGamma_{t,s}\;dW_s,\quad s\in [t,T],\quad \varGamma_{t,t} = 1.
\end{align}
Moreover, a solution $\hat{\mu}$ of the optimal control problem \eqref{eq:dual}--\eqref{eq:densityprocess} exists.
\end{theorem}

The following lemma provides a useful estimate.

\begin{lemma}
\label{c:Gamma-bound}
A constant $C$ exists, such that for all $0\le t < s \le T$ and all $\{\varGamma_{t,s}\}$ that satisfy \eqref{eq:densityprocess}, we have
\begin{align}
\|\varGamma_{t,s} -1\|^2 \le C(s-t).
\end{align}
\end{lemma}
\begin{proof}
Using It\^{o} isometry, we obtain the chain of relations
\begin{align}
\|\varGamma_{t,s} - 1\|^2 = \int_t^s \|\mu_r\varGamma_{t,r}\|^2\,dr\le \int_t^s \|\mu_r\|^2\|\varGamma_{t,r}\|^2\,dr \le \int_t^s \|\mu_r\|^2\big( 1+ \|\varGamma_{t,r}-1\|^2\big)\,dr.
\end{align}
If $K$ is a uniform upper bound on the norm of the subgradients of $g(r,0)$ we deduce that $\|\varGamma_{t,r} - 1\|^2 \le \psi_{r}$, $r\in [t, T]$, where $\psi$ satisfies the ODE:
$\frac{d\psi_r}{dr} = K^2(1+\psi_r)$, with $\psi_t=0$. Consequently,
\begin{align}
\|\varGamma_{t,r} - 1\|^2 \le \psi_r =  e^{K^2(r-t)} -1.
\end{align}
The convexity of the exponential function yields the postulated bound. 
\end{proof}

\vspace{1ex}
The dual representation theorem allows us to transform the  risk evaluation problem to a stochastic control problem. Our objective now is to approximate the evaluation of risk on a short interval $\Delta$ so as to reduce functional optimization to vector optimization. To proceed, we investigate the corresponding maximum principle of the control problem~\eqref{eq:dual}.

\section{Stochastic Maximum Principle}
\label{s:maximum-principle}
\label{sec:3.1}
In this section, we decipher the optimality conditions of the dual stochastic control problem
\eqref{eq:dual}--\eqref{eq:densityprocess}. Since only the process $\{\varGamma_{t,s}\}_{s\in[t,T]}$  is controlled, the analysis is rather standard.


Suppose $\hat{\mu}$ is the optimal control; then, for any $\mu\in\mathcal{A}_{0,T}$ and $0\leq \alpha \leq 1$, we can form a perturbed control function
\[
\mu^\alpha =  \hat{\mu} + \alpha(\mu - \hat{\mu}).
\]
It is still an element of $\mathcal{A}_{0,T}$, due to the convexity of the sets $A_s$. The processes $\hat{\varGamma}$, $\varGamma$ and $\varGamma^\alpha$ are the state processes under the controls $\hat{\mu}$, $\mu$, and $\mu_\alpha$, respectively.

We linearize the state equation \eqref{eq:densityprocess} about $\hat{\varGamma}$ to get, for $s\in[t,T]$,
\begin{align}\label{eq:linearize}
d\eta_s^\mu = \big[\hat{\mu}_s\eta_s^\mu + \hat\varGamma_{t,s}(\mu_s - \hat{\mu}_s)\big]\;dW_s,\quad \eta^\mu_t = 0.
\end{align}
It is evident that this equation has a unique strong solution  $\eta^\mu\in \Hc^2[t,T]$. Denote
\[
{h}_s^\alpha = \frac{1}{\alpha}\big[ \varGamma^\alpha_{t,s} - \hat\varGamma_{t,s}\big] - \eta_s^\mu,\quad s\in[t,T].
\]
The usefulness of the linearized equation \eqref{eq:linearize} is justified by the following standard result.
\begin{lemma}
\label{lem:est1}
\begin{equation}
\label{hto0}
\lim_{\alpha\rightarrow 0}\sup_{t\le s\le T}\|{h}_s^\alpha\|^2 = 0.
\end{equation}
\end{lemma}
\begin{proof}
We first prove that
\begin{equation}
\label{Gamma-conv}
\lim_{\alpha\rightarrow 0}\sup_{t\le s\le T}\|\varGamma^\alpha_{t,s} - \hat\varGamma_{t,s}\|^2 = 0.
\end{equation}
We have
\begin{equation}
\label{dGamma-dif}
d\big(\varGamma^\alpha_{t,s} - \hat\varGamma_{t,s}\big) = \big(\mu^\alpha_s \varGamma^\alpha_{t,s} - \hat{\mu}_s\hat\varGamma_{t,s}\big)\;dW_s = \big((\mu^\alpha_s -\hat{\mu}_s) \hat\varGamma_{t,s} + {\mu}^\alpha_s(\varGamma^\alpha_{t,s}-\hat\varGamma_{t,s})\big)\;dW_s.
\end{equation}
By It\^{o} isometry,
\begin{multline*}
 \|\varGamma^\alpha_{t,r} - \hat{\varGamma}_{t,r} \|^2  = \int_t^r\| (\mu^\alpha_s -\hat{\mu}_s) \hat\varGamma_{t,s} + {\mu}^\alpha_s(\varGamma^\alpha_{t,s}-\hat\varGamma_{t,s})\|^2 \;ds \\
\le 2\int_t^r\| (\mu^\alpha_s -\hat{\mu}_s) \hat\varGamma_{t,s}\|^2\;ds + 2\int_t^r  \| {\mu}^\alpha_s(\varGamma^\alpha_{t,s}-\hat\varGamma_{t,s})\|^2 \;ds \\
\le 2\int_t^r\| (\mu^\alpha_s -\hat{\mu}_s) \hat\varGamma_{t,s}\|^2\;ds + K \int_t^r  \|\varGamma^\alpha_{t,s}-\hat\varGamma_{t,s}\|^2 \;ds,
\end{multline*}
where $K$ is a constant. Since the first integral on the right hand side converges to 0, as $\alpha\to 0$,
the Gronwall inequality yields \eqref{Gamma-conv}.

We can now prove the assertion of the lemma.
Combining \eqref{dGamma-dif} and \eqref{eq:linearize}, we obtain the stochastic differential equation for ${h}^\alpha$:
\begin{align*}
d{h}_s^\alpha &= 
\bigg\{ \frac{1}{\alpha}\big[(\hat{\mu}_s + \alpha(\mu_s - \hat{\mu}_s))\varGamma^\alpha_{t,s}-\hat{\mu}_s\hat\varGamma_{t,s}\big]-\hat{\mu}_s\eta_s^\mu - \hat\varGamma_{t,s}(\mu_s - \hat{\mu}_s)\bigg\}\;dW_s\\
&= \bigg\{ \frac{1}{\alpha}\hat{\mu}_s\big[ \varGamma^\alpha_{t,s}-\hat\varGamma_{t,s}\big] + (\mu_s-\hat{\mu}_s)\big[\varGamma^\alpha_{t,s}-\hat\varGamma_{t,s}\big] - \hat{\mu}_s\eta_s^\mu \bigg\}\;dW_s
= \bigg\{ \hat{\mu}_s {h}_s^\alpha + (\mu_s-\hat{\mu}_s)\big[\varGamma^\alpha_{t,s}-\hat\varGamma_{t,s}\big]\bigg\}\;dW_s.
\end{align*}
Since the processes $\{\hat{\mu}_s\}$ and $\{\mu_s\}$ are bounded, It\^o isometry yields again
\[
\|{h}_r^\alpha\|^2 \leq K\int_t^{r}  \|{h}_s^\alpha\|^2\;ds
+ K\int_t^{r} \|\varGamma^\alpha_{t,s}-\hat\varGamma_{t,s}\|^2\;ds,
\]
where $K$ is constant. By the Gronwall inequality, using \eqref{Gamma-conv}, we get the desired result.
\end{proof}

The convergence result above directly leads to the following {variational inequality}.

\vspace{1ex}
\begin{lemma}\label{lem:vi} For any $\mu\in \mathcal{A}_{0,T}$ we have
\begin{align}
\label{eq:vi}
\mathbb{E}\big[\,\xi_T\eta_{T}^\mu\,\big]\leq 0.
\end{align}
\end{lemma}
\begin{proof}
Since $\hat{\mu}$ is the optimal control,
\[
 \mathbb{E}\big[\,\xi_T\big(\,\varGamma^\alpha_{t,T} - \hat{\varGamma}_{t,T}\,\big)\,\big]\leq 0.
\]
Relation \eqref{hto0} leads to
\[ \lim_{\alpha\rightarrow 0}\mathbb{E}\Big[\xi_T\frac{1}{\alpha}\big(\varGamma^\alpha_{t,T} - \hat{\varGamma}_{t,T}\big)\Big]
=   \mathbb{E}\big[\xi_T\eta_{T}^\mu\big]\leq 0,
\]
as required.
\end{proof}

We now express the expected value in \eqref{eq:vi} as an integral, to obtain a pointwise variational inequality (the maximum principle). To this end, we introduce the following backward stochastic differential equation (the \emph{adjoint equation}):
\begin{equation}
\label{eq:adjoint}
dp_s =  -\hat{\mu}_s k_s  \,ds + k_s\,dW_s,\,\,p_{T} = \xi_T,\,\,s\in[t,T],
\end{equation}
with $\xi_T = \varPhi\big(X_T^{t,x}\big)$. Recall that $\mu k$ is understood as the scalar product of $\mu$ and $k$.
 The processes $\{p_s\}_{t\le s \le T}$ and $\{k_s\}_{t\le s \le T}$ are
elements of the spaces $\Hc^2[t,T]$ and $\Hc^{2, n}[t, T]$, respectively.

By construction,
$\mathbb{E}\big[\,\xi_T\eta_{T}^\mu\,\big]=\mathbb{E}\big[\,{p}_T\eta_{T}^\mu\,\big]$.
 Consider the product process $\{p_s\eta_s^\mu\}_{s\in[t,T]}$ as in \cite[Cor. 5.6]{XYZ}.
Applying the generalized It\^o formula \cite[Thm.II.5.1]{Ikeda-Watanabe} to this process, we obtain
\[
d(p_s\eta_s^\mu) =  \Big({k}_s\eta_s^\mu + {p}_s
\big[\hat{\mu}_s\eta_s^\mu + \hat{\varGamma}_{t,s}(\mu_s - \hat{\mu}_s)\big]\Big)\;dW_s
 + {k}_s  \hat{\varGamma}_{t,s}(\mu_s - \hat{\mu}_s)\;ds.
\]
If follows that
\begin{equation}
\label{vi:int}
\mathbb{E}\big[\xi_{T}\eta_{T}^\mu\big] = \mathbb{E}\bigg[\,\int_t^{T}  {k}_s\hat{\varGamma}_{t,s}(\mu_s - \hat{\mu}_s)ds\,\bigg].
\end{equation}
We can summarize our derivations in the following version of the maximum principle.
We define the Hamiltonian $H: \Rb\times\Rb^n\times\Rb^n\to\Rb$:
\[
H(\gamma, {a}, \kappa)= \gamma\kappa {a}.
\]

\vspace{1ex}
\begin{theorem}
\label{t:max-principle}
For almost all $s\in [t,T]$, with probability 1,
\[
H( \hat{\varGamma}_{t,s}, \hat{\mu}_s, {k}_s)  = \max_{{a}\in A_s} H( \hat{\varGamma}_{t,s}, {a}, {k}_s).
\]
\end{theorem}
\begin{proof}
For any $\mu\in \mathcal{A}_{0,T}$, we define the set
\[
\mathcal{G} = \{\,(\omega, s)\in\varOmega\times [t,T] : {k}_s\hat{\varGamma}_{t,s}(\mu_s - \hat{\mu}_s) > 0\,\}.
\]
We construct a new control $\mu^*\in \mathcal{A}_{0,T}$:
\[
\mu^*_s =
\begin{cases}
\mu_s,       & \  (\omega,s)\in \mathcal{G},\\
\hat{\mu}_s,  & \  \text{otherwise.}
\end{cases}
\]
The measurability and adaptedness of $\mu^*$ can be easily verified. It follows from \eqref{eq:vi} and \eqref{vi:int} that
\[
\mathbb{E}\bigg[\,\int_t^{T}{k}_s\hat{\varGamma}_{t,s}(\mu^*_s - \hat{\mu}_s)\;ds\,\bigg]
\le 0.
\]
Then, by the construction of $\mu^*$,
\[
\iint_{\mathcal{G}} {k}_s\hat{\varGamma}_{t,s}(\mu^*_s - \hat{\mu}_s)\;ds\;\mathbb{P}(d\omega)\leq 0.
\]
Since the integrand is positive on $\mathcal{G}$, the product measure of $\mathcal{G}$ must be zero.
\end{proof}

\begin{remark}
\label{r:dual-dual}
Observe that due to Theorem \ref{t:max-principle}, we have
\[
\hat{\mu}_s k_s  = \max_{{a} \in A_s} {a} k_s = g(s,k_s).
\]
Therefore,  $p\equiv Y$ and $k\equiv Z$ solve the adjoint equation \eqref{eq:adjoint}, which becomes identical to the BSDE \eqref{BSDE}.
The optimal dual control is
\[
\hat{\mu}_s  = \argmax_{{a} \in A_s}\,{a} Z_{s}.
\]
\end{remark}

\section{Error Estimates for Constant Controls on Small Intervals}
\label{s:estimates}

To reduce an infinite dimensional control problem to a finite dimensional vector optimization,  we partition the interval $[0,T]$ into $N$ short pieces of length $\Delta  = T/N$, and  develop a scheme for evaluating the risk measure \eqref{rho} by using constant dual controls
on each piece. We denote $t_i = i\Delta  $, for $i=0,1, \dots , N$.

For simplicity, in addition to Assumption \ref{a:riskmeasure}, we assume that the driver $g$ does not depend on time, and thus all sets $A_t = \partial g(0)$ are the same. We denote them with the symbol~$A$; as we shall see later, it is not a major restriction.

If the system's state at time $t_i$ is $x$,  then the value of the risk measure \eqref{risk-value-function} is then the optimal value of problem \eqref{eq:dual}.
By dynamic programming,
\[
v(t_i,x) = \rho^x_{t_i,t_{i+1}}\big[ v\big(t_{i+1},X_{t_{i+1}}^{t_i,x}\big)\big].
\]
The risk measure $\rho^x_{t_i,t_{i+1}}[\,\cdot\,]$ is defined by problem \eqref{eq:dual}, with
terminal time $t_{i+1}$ and the function $\varPhi(\cdot)$ replaced by $v(t_{i+1},\,\cdot\,)$.
Equivalently, it is equal to $Y_{t_i}^{t_i,x}$, in the corresponding forward--backward system
on the interval $[t_i,t_{i+1}]$:
\begin{align}
dX_s^{t_i,x} &= b(s,X_s^{t_i,x})\;ds + \sigma(s,X_s^{t_i,x})\; dW_s,\quad X_{t_i}^{t_i,x} = x,
\label{SDEtxi}\\
dY_s^{t_i,x} &= -g(Z_s^{t_i,x})\;ds + Z_s^{t_i,x}\;dW_s, \quad  Y_{t_{i+1}}^{t_i,x} = v\big(t_{i+1},X_{t_{i+1}}^{t_i,x}\big).\label{BSDEtxi}
\end{align}
Under Assumption \ref{a:c01b}, the function $v(\cdot,\cdot)$ is the  viscosity solution of the associated
Hamilton--Jacobi--Bellman equation:
\begin{equation}
\label{HJB-full}
v_t(t,x) + v_x(t,x)b(t,x)  + {\half} \text{\rm tr}\big( v_{xx}(t,x)\sigma(t,x)\sigma^T(t,x)\big)+g\big(v_x(t,x)\sigma(t,x)\big) = 0,
\end{equation}
with the terminal condition $v(T,x)=\varPhi(x)$;   see, \emph{e.g.}, \cite[sec. 5.5]{zjf-book}.

Suppose
we use a constant control in the interval $[t_i,t_{i+1}]$:
\begin{align}\label{eq:def}
\mu_s := \hat{\mu}_{t_i} = \argmax_{{a} \in A}\,{a}  Z_{t_i}^{t_i,x} ,\quad\forall s\in[t_i,t_{i+1}],
\end{align}
where $(Y^{t_i,x},Z^{t_i,x})$ solve the BSDE \eqref{BSDEtxi} in this interval (which is the adjoint equation \eqref{eq:adjoint}).
We still use $\hat{\Gamma}$ to denote the state evolution under the optimal control, while  $\Gamma$ is the process under control $\mu$ defined in \eqref{eq:def}.

Our objective is to show that a constant $C$ exists, independent of $x$, $N$, and $i$, such that
the approximation error on the $i$th interval can be bounded as follows:
 \begin{equation}
 \label{approx-error}
 0 \le \Eb\big[v\big(t_{i+1},X_{t_{i+1}}^{t_i,x}\big)
 \big(\hat{\Gamma}_{t_{i},t_{i+1}}-\Gamma_{t_{i},t_{i+1}}\big)\big]
 \le C \Delta^{\frac{3}{2}}.
 \end{equation}
The fact that we do not know $Z_{t_i}^{t_i,x}$ will not be essential;
later, we shall generate even better constant controls by discrete-time dynamic programming.

We can now derive some useful estimates for the constant control function \eqref{eq:def}.

\begin{lemma}\label{lem:estimates1} A constant $C$ exists, such that for all $x$, $N$ and $i$
\begin{align}
\label{eq:est1}
\mathbb{E}\bigg[v\big(t_{i+1},X_{t_{i+1}}^{t_i,x}\big)\int_{t_{i}}^{t_{i+1}}(\hat{\mu}_s - \mu_s)\hat{\Gamma}_{t_{i},s}\;dW_s\,\bigg] \le C \Delta^{\frac{3}{2}}.
\end{align}
\end{lemma}

\begin{proof}
For simplicity, we write $Y_s$ for $Y^{t_i,x}_s$ and $Z_s$ for $Z^{t_i,x}_s$. From \eqref{BSDEtxi} we get:
\begin{equation}
\label{xiT-integral}
v\big(t_{i+1},X_{t_{i+1}}^{t_i,x}\big) = Y_{t_{i}}
- \int_{t_i}^{t_{i+1}}  g(Z_s) \,ds + \int_{t_i}^{t_{i+1}} Z_s\,dW_s.
\end{equation}
Then the left hand side of \eqref{eq:est1} can be written as follows:
\begin{equation}
\label{two-terms}
\begin{aligned}
\lefteqn{\mathbb{E}\bigg[\Big(Y_{t_{i}} - \int_{t_{i}}^{t_{i+1}} g(Z_t)\;dt + \int_{t_{i}}^{t_{i+1}} Z_t\;dW_t\Big)\int_{t_{i}}^{t_{i+1}}(\hat{\mu}_s - \mu_s)\hat{\Gamma}_{t_{i},s}\;dW_s\,\bigg]}\\
&=-\Eb\bigg[ \int_{t_{i}}^{t_{i+1}} g(Z_t)\;dt \int_{t_{i}}^{t_{i+1}}(\hat{\mu}_s - \mu_s)\hat{\Gamma}_{t_{i},s}\;dW_s\,\bigg]
 +  \Eb\bigg[ \int_{t_{i}}^{t_{i+1}} Z_t\;dW_t \int_{t_{i}}^{t_{i+1}}(\hat{\mu}_s - \mu_s)\hat{\Gamma}_{t_{i},s}\;dW_s\,\bigg].
\end{aligned}
\end{equation}
The first term on the right hand side of \eqref{two-terms} can be bounded by the Cauchy-Schwarz inequality and the It\^{o} isometry:
\begin{align*}
\lefteqn{-\Eb\bigg[ \int_{t_{i}}^{t_{i+1}} g(Z_t)\;dt\int_{t_{i}}^{t_{i+1}}(\hat{\mu}_s - \mu_s)\hat{\Gamma}_{t_{i},s}\;dW_s\,\bigg]}\quad\\
&\le
\left(\Eb\bigg[ \bigg(\int_{t_{i}}^{t_{i+1}} g(Z_s)\;ds\bigg)^2\bigg]\right)^{\frac{1}{2}}\left( \Eb\bigg[ \bigg(\int_{t_{i}}^{t_{i+1}} (\hat{\mu}_s - \mu_s)\hat{\Gamma}_{t_{i},s}\;dW_s\bigg)^2\bigg]\right)^{\frac{1}{2}}\\
 & =
\bigg\|\int_{t_{i}}^{t_{i+1}} g(Z_s)\;ds\bigg\| \; \bigg(\int_{t_{i}}^{t_{i+1}} \Eb\Big[ (\hat{\mu}_s - \mu_s)^2\hat{\Gamma}_{t_{i},s}^2\Big] \;ds\bigg)^{\frac{1}{2}}\\
&\le
 C_1 \int_{t_{i}}^{t_{i+1}} \| Z_s\|\;ds \; \bigg(\int_{t_{i}}^{t_{i+1}} \Eb\Big[ (\hat{\mu}_s - \mu_s)^2\hat{\Gamma}_{t_{i},s}^2\Big] \;ds\bigg)^{\frac{1}{2}}
\le  C_2 \Delta^{\frac{3}{2}},
\end{align*}
where $C_1$ and $C_2$ are some constants.  In the last inequality, we used the fact that $g(\cdot)$ is Lipschitz and
a uniform bound on $\|Z_s\|$ exists; see, \emph{e.g.}, \cite[sec. 5.2]{zjf-book}.

The second term on the right hand side of \eqref{two-terms} can be estimated in a similar way:
\begin{align*}
\lefteqn{\Eb\bigg[ \int_{t_{i}}^{t_{i+1}} Z_t\;dW_t \int_{t_{i}}^{t_{i+1}}(\hat{\mu}_s - \mu_s)\hat{\Gamma}_{t_{i},s}\;dW_s\,\bigg]
=
\Eb\bigg[ \int_{t_{i}}^{t_{i+1}} Z_s(\hat{\mu}_s - \mu_s)\hat{\Gamma}_{t_{i},s}\;ds\,\bigg]}\quad\\
&\le
 \bigg(\int_{t_{i}}^{t_{i+1}}\Eb\big[ | Z_s|^2\big]\;ds\bigg)^{\frac{1}{2}} \; \bigg(\int_{t_{i}}^{t_{i+1}} \Eb\Big[ (\hat{\mu}_s - \mu_s)^2\hat{\Gamma}_{t_{i},s}^2\Big] \;ds\bigg)^{\frac{1}{2}}
\le C_3 \Delta^{\frac{3}{2}},
\end{align*}
where $C_3$ is a sufficiently large constant.
\end{proof}
\vspace{1ex}

We also have the estimate below:

\vspace{1ex}
\begin{lemma}\label{lem:estimates2} A constant $C$ exists, such that for all $x$,  $N$, and $i$
\begin{align}\label{eq:est2}
\mathbb{E}\bigg[v\big(t_{i+1},X_{t_{i+1}}^{t_i,x}\big)\int_{t_{i}}^{t_{i+1}}(\hat{\Gamma}_{t_{i},s} - \Gamma_{t_{i},s})\mu_s\;dW_s\bigg]   \le C \Delta^{\frac{3}{2}}.
\end{align}
\end{lemma}

\begin{proof}
We proceed as in the proof of the previous lemma. We use \eqref{xiT-integral} and express the left hand side of \eqref{eq:est2} as follows:
\begin{equation}
\label{two-terms2}
\begin{aligned}
\lefteqn{\mathbb{E}\bigg[v\big(t_{i+1},X_{t_{i+1}}^{t_i,x}\big)\int_{t_{i}}^{t_{i+1}}(\hat{\Gamma}_{t_{i},s} - \Gamma_{t_{i},s})\mu_s\;dW_s\bigg]}\\
&=-\Eb\bigg[ \int_{t_{i}}^{t_{i+1}} g(Z_t)\;dt \int_{t_{i}}^{t_{i+1}}(\hat{\Gamma}_{t_{i},s} - \Gamma_{t_{i},s})\mu_s\;dW_s\,\bigg]
+  \Eb\bigg[ \int_{t_{i}}^{t_{i+1}} Z_t\;dW_t \int_{t_{i}}^{t_{i+1}}(\hat{\Gamma}_{t_{i},s} - \Gamma_{t_{i},s})\mu_s\;dW_s\,\bigg].
\end{aligned}
\end{equation}
The first term on the right hand side of \eqref{two-terms2} can be dealt with by the Cauchy-Schwarz inequality and {It\^o isometry}, exactly as before:
\begin{align*}
\lefteqn{\left| \Eb\bigg[ \int_{t_{i}}^{t_{i+1}} g(Z_s)\;ds \int_{t_{i}}^{t_{i+1}}(\hat{\Gamma}_{t_{i},s} - \Gamma_{t_{i},s})\mu_s\;dW_s\,\bigg]\,\right|}\\
&\le \left(\Eb\bigg[ \bigg(\int_{t_{i}}^{t_{i+1}} g(Z_s)\;ds\bigg)^2\bigg]\right)^{\frac{1}{2}}
\left(\int_{t_{i}}^{t_{i+1}}\Eb\big[ (\hat{\Gamma}_{t_{i},s} - \Gamma_{t_{i},s})^2|\mu_s|^2\big]\;ds
\right)^{\frac{1}{2}} \le  C_1 \Delta^{\frac{3}{2}}.
\end{align*}
To estimate the second term, consider two controlled dual processes:
\[
\hat{\Gamma}_{t_{i},t} = 1 + \int_{t_{i}}^{t} \hat{\mu}_s\hat{\Gamma}_{t_{i},s}\;dW_s,\qquad
\Gamma_{t_{i},t} = 1 + \int_{t_{i}}^{t}\mu_s\Gamma_{t_{i},s}\;dW_s.
\]
Taking the difference yields,
\begin{equation}
\label{Gamma-dif}
\hat{\Gamma}_{t_{i},t} - \Gamma_{t_{i},t} = \int_{t_{i}}^{t} (\hat{\mu}_s - \mu_s)\hat{\Gamma}_{t_{i},s}\;dW_s + \int_{t_{i}}^{t}(\hat{\Gamma}_{t_{i},s} - \Gamma_{t_{i},s})\mu_s\;dW_s.
\end{equation}
By It\^{o} isometry,
\[
\Eb\big[ (\hat{\Gamma}_{t_{i},t} - \Gamma_{t_{i},t})^2\big]  \le
2 \int_{t_{i}}^{t} \Eb\big[(\hat{\mu}_s - \mu_s)^2\hat{\Gamma}_{t_{i},s}^2\big]\;ds
+ 2 \int_{t_{i}}^{t}\Eb\big[(\hat{\Gamma}_{t_{i},s} - \Gamma_{t_{i},s})^2\mu_s^2\big]\;ds.
\]
By Lemma \ref{c:Gamma-bound} and the boundedness of the processes $\{\mu_t\}$ and $\{\hat{\mu}_t\}$, a constant $C_2$ exists, such that
\[
\Eb\big[ (\hat{\Gamma}_{t_{i},t} - \Gamma_{t_{i},t})^2\big] \le C_2|t-t_{i}|.
\]
Thus, we can write the bound:
\begin{multline*}
\left| \Eb\bigg[ \int_{t_{i}}^{t_{i+1}} Z_t\;dW_t \int_{t_{i}}^{t_{i+1}}(\hat{\Gamma}_{t_{i},s} - \Gamma_{t_{i},s})\mu_s\;dW_s\,\bigg]\,\right|
=\left| \Eb\bigg[ \int_{t_{i}}^{t_{i+1}} Z_s(\hat{\Gamma}_{t_{i},s} - \Gamma_{t_{i},s})\mu_s\;ds\bigg]\,\right|\\
 \le C_1 \left(\int_{t_{i}}^{t_{i+1}} \Eb\big[|Z_s|^2\big]\;ds\right)^{\frac{1}{2}}
\left(\int_{t_{i}}^{t_{i+1}} \Eb\big[ (\hat{\Gamma}_{t_{i},s}-\Gamma_{t_{i},s})^2\big] \;ds\right)^{\frac{1}{2}} \le  C \Delta^{\frac{3}{2}},
\end{multline*}
where $C$ is a sufficiently large constant.
\end{proof}
\vspace{1ex}

We can now compare the value of the functional \eqref{eq:dual} with the value achieved by a constant control $\mu$.

\vspace{1ex}
\begin{theorem}\label{thm:main}
Suppose Assumptions \ref{a:c01b}
and \ref{a:riskmeasure} are satisfied. Then
a constant $C$ exists, independent on $x$, $N$ and $i$, such that inequality \eqref{approx-error} holds.
\end{theorem}
\begin{proof}
Using \eqref{Gamma-dif}, we obtain
\begin{multline*}
\Eb\big[v\big(t_{i+1},X_{t_{i+1}}^{t_i,x}\big)
 \big(\hat{\Gamma}_{t_{i},t_{i+1}}-\Gamma_{t_{i},t_{i+1}}\big)\big]\\
 =\Eb\left[v\big(t_{i+1},X_{t_{i+1}}^{t_i,x}\big)\int_{t_{i}}^{t_{i+1}} (\hat{\mu}_s - \mu_s)\hat{\varGamma}_{t_i,s}\;dW_s\right]
 + \Eb\left[v\big(t_{i+1},X_{t_{i+1}}^{t_i,x}\big)\int_{t_{i}}^{t_{i+1}}(\hat{\varGamma}_{t_i,s} - \varGamma_{t_i,s})\mu_s\;dW_s\right].
\end{multline*}
Combining the estimates from Lemmas \ref{lem:estimates1} and \ref{lem:estimates2}, we obtain the
postulated result.
\end{proof}
\vspace{1ex}

 An even smaller error than \eqref{approx-error} can be achieved by choosing the \emph{best} constant control in the interval $[t_i,t_{i+1}]$. For a constant $\mu_t\equiv{a}$, where ${a}\in A$, the dual state equation \eqref{eq:densityprocess} has a closed-form solution, the exponential martingale:
\[
\varGamma_{t_i,t} = \exp\Big({a} (W_t-W_{t_i}) -\frac{t-t_i}{2}|{a}|^2\Big), \quad t\in [t_i,t_{i+1}].
\]
It follows that an $\mathcal{O}\big(\Delta^{\frac{3}{2}}\big)$ approximation of the risk measure can be obtained by solving the following simple vector optimization problem:
\begin{equation}
\label{tilde-rho}
\tilde{\rho}^x_{t_i,t_{i+1}}\big[ v\big(t_{i+1},X_{t_{i+1}}^{t_i,x}\big)\big] := \max_{{a}\in A}\Eb\Big[v\big(t_{i+1},X_{t_{i+1}}^{t_i,x}\big) \exp\Big({a} (W_{t_{i+1}}-W_{t_i}) -\frac{\Delta}{2}|{a}|^2\Big)\Big].
\end{equation}
Opposite to \eqref{eq:def}, we do not need to know $Z_{t_i}$ to solve this problem.

By Theorem \ref{thm:main},
\begin{equation}
\label{one-step-error}
 0 \le   v(t_{i},x) -
\tilde{\rho}^x_{t_i,t_{i+1}}\big[ v\big(t_{i+1},X_{t_{i+1}}^{t_i,x}\big)\big] \le C\Delta^{\frac{3}{2}}.
\end{equation}

By construction, the approximating measure of risk $\tilde{\rho}^x_{t_i,t_{i+1}}[\,\cdot\,]$ is coherent and satisfies all properties  (i)-(iii) of
Theorem~\ref{t:geval-prop}.

\section{Discrete-Time Approximations by Dynamic Programming}
\label{s:discrete-approximation}

The time-consistency of dynamic risk measure leads to the nested form below:
\begin{equation}
\label{eq:composition}
\rho_{0, T}\big[\,\varPhi(X_T)\,\big] = \rho_{t_0, t_1}\bigg[\rho_{t_1, t_2}\Big[ \dots
\rho_{t_{N-2}, t_{N-1}}\Big[\rho_{t_{N-1}, t_N}\big[\varPhi(X_T)\big]\Big]\dots\Big]\bigg].
\end{equation}
By using optimal constant dual controls on each interval $[t_i,t_{i+1})$, we may approximate this composition by dynamic programming. For $i=N$ we define $\tilde{v}_N(x)=\varPhi(x)$. Then, for $i=N-1,N-2,\dots,0$, and for $x\in \Rb^n$, we restart the diffusion \eqref{SDE} from $x$ at time $t_i$ as in \eqref{SDEtxi}.
Having obtained $X^{t_i,x}_{t_{i+1}}$, we can calculate the approximate risk measure \eqref{tilde-rho} on the interval
$[t_i,t_{i+1}]$:
\begin{equation}
\label{tildeV}
\tilde{v}_i(x) = \tilde{\rho}^x_{t_{i}, t_{i+1}}\big[\tilde{v}_{i+1}\big(X^{t_i,x}_{t_{i+1}}\big)\big]
=\max_{{a}\in A}\Eb\Big[\tilde{v}_{i+1}\big(X^{t_i,x}_{t_{i+1}}\big)  \exp\Big({a} (W_{t_{i+1}}-W_{t_i}) -\frac{\Delta}{2}|{a}|^2\Big)\Big].
\end{equation}

\begin{theorem}
Suppose Assumptions \ref{a:c01b} and \ref{a:riskmeasure} are satisfied. Then a constant $C$ exists, such that
for all $N$ and $x$ we have:
\begin{equation}
\label{error-bounds}
 0 \le  v(t_i,x) - \tilde{v}_i(x) \le C (N-i) \Delta^{\frac{3}{2}},\quad i=0,1,\dots,N.
\end{equation}
In particular,
$ 0 \le  v(0,x) -\tilde{v}_0(x) \le CT \Delta^{\frac{1}{2}}$.
\end{theorem}
\begin{proof}
The result follows by backward induction. It is obviously true for $i=N$. If it is true for
$i+1$, we can easily verify it for $i$. By the translation property of
$\tilde{\rho}^x_{t_{i}, t_{i+1}}[\,\cdot\,]$ and \eqref{one-step-error} we obtain:
\begin{align*}
v(t_{i},x) - \tilde{v}_i(x)
&=
v(t_{i},x) - \tilde{\rho}^x_{t_{i}, t_{i+1}}\big[\tilde{v}_{i+1}\big(X^{t_i,x}_{t_{i+1}}\big)\big] \\
&\le
v(t_{i},x) - \tilde{\rho}^x_{t_i,t_{i+1}}\big[ v\big(t_{i+1},X_{t_{i+1}}^{t_i,x}\big)\big] + C \Delta^{\frac{3}{2}}
\le C (N-i) \Delta^{\frac{3}{2}},
\end{align*}
as required.  Nonnegativity of the error follows from the fact that piecewise-constant control is suboptimal in the dual problem.
\end{proof}

In practice, the forward process \eqref{SDEtx} is simulated in an approximate way, for example, by \emph{Euler's method}:
\begin{equation}
\label{SDEtx-app}
\tilde{X}^{t_i,x}_{t_{i+1}} \approx x + b(t_i,x)\,\Delta + \sigma(t_i,x)\, \Delta W,\  \Delta W \sim N(0, \Delta\mathbb{I}).
\end{equation}
It is well known that for small $\Delta$, the error of this Euler scheme is $\mathcal{O}\big(\Delta^\frac{1}{2}\big)$. Since $\tilde{X}^{t_i,x}_{t_{i+1}}$ is a normal random vector, streamlined calculation of the risk measure is possible. Denoting by $\Nc$  a standard normal random vector
with independent components, we can simplify the calculation
of the risk measure in \eqref{tildeV} as follows:
\begin{equation}
\label{tildeVapp}
\tilde{v}_i(x) \approx
\max_{{a}\in A}\Eb\Big[\tilde{v}_{i+1}\big(x+b(t_i,x)\Delta+ \sigma(t_i,x)\Nc \Delta^{\frac{1}{2}}  \big)
\exp\Big({a} \Nc \Delta^{\frac{1}{2}}-\half|{a}|^2\Delta\Big)\Big].
\end{equation}
Observe that the same normal random vector $\Nc$ is used in both terms of this expression.

\vspace{1ex}
\begin{remark} Our earlier assumption of time-homogeneity of $g$  is barely a restriction after discretization, because $g$ can be $\frac{1}{2}$-H\"older continuous between the grid points. As long as the risk aversion does not change abruptly, the numerical method developed can be easily adapted to the case of a time-dependent driver.
\end{remark}

\begin{remark} Following an insightful observation of a Referee, we provide a further approximation of \eqref{tildeVapp} for  small $\Delta$. The main purpose of this approximation is to gain insight into the virtues
of our method. Expanding $\tilde{v}_{i+1}(\cdot)$ at $x$ and the exponent function at 0,
and skipping terms of order higher than $\Delta$, we obtain
\begin{align*}
\tilde{v}_i(x) &\approx
\max_{{a}\in A}\Eb\Big[\Big(\tilde{v}_{i+1}\big(x) +  \big(\tilde{v}_{i+1}\big)_{\!x} (x)\,\big(b(t_i,x)\Delta+ \sigma(t_i,x)\Delta^{\frac{1}{2}}\Nc\big)\\
 &\hspace{5
 em} +  {\half}\Nc^T\sigma(t_i,x)^T\big(\tilde{v}_{i+1}\big)_{xx} (x)\,\sigma(t_i,x)\Nc  \Delta \Big)
 \Big( 1 + {a} \Delta^{\frac{1}{2}}  \Nc   -\frac{\Delta}{2}|{a}|^2 + {\half} ({a}\Nc)^2 \Delta \Big)
 \Big] \\
 &\approx \tilde{v}_{i+1}\big(x) + \big(\tilde{v}_{i+1}\big)_{\!x} (x) \,b(t_i,x)\Delta
 + {\half} \text{\rm tr}\big[\big(\tilde{v}_{i+1}\big)_{\!xx} (x)\,\sigma(t_i,x)\,\sigma^T(t_i,x)\big]\Delta\\
  & \qquad + \max_{{a}\in A}\Big[ \big(\tilde{v}_{i+1}\big)_{\!x}\sigma(t_i,x){a}^T \Delta \Big].
\end{align*}
With a view at \eqref{dual-driver}, the last equation can be rewritten as follows
\[
 \frac{\tilde{v}_{i}\big(x) - \tilde{v}_{i+1}(x)}{\Delta} \approx
 \big(\tilde{v}_{i+1}\big)_{\!x} (x) \,b(t_i,x)
 + {\half} \text{\rm tr}\big[\big(\tilde{v}_{i+1}\big)_{\!xx} (x) \,\sigma(t_i,x) \,\sigma^T(t_i,x)\big]
 + g\big[ \big(\tilde{v}_{i+1}\big)_{\!x}(x)\,\sigma(t_i,x) \big].
\]
We see that our recursive method \eqref{tildeVapp}  follows the HJB equation \eqref{HJB-full} very closely, but we can implement it
without assuming the existence of the derivatives of the approximate value function or using them in any direct or implicit way.
\end{remark}

\section{Applications}
\label{s:risk-management}

We present two numerical examples illustrating the application and performance of the dual method.

\subsection{Two-Stage Valuation}

After the credit crunch, the management of risk is an increasingly important function of any financial institution. The primary goal is to have sufficient capital reserves against potential losses in the future. Such risk management is divided into two stages: \emph{scenario generation} and \emph{portfolio re-pricing}.

Scenario generation refers to the construction of sample paths over a given time horizon. This is also called the \emph{outer stage}, where {Monte Carlo simulation} is used to generate paths
of systems governed by stochastic differential equations. Repricing of portfolio amounts to the computation of the portfolio value at a certain time horizon. The portfolio may consist of derivative securities with nonlinear payoffs that, in conjunction with financial models, require Monte Carlo simulation for this \emph{inner stage} as well (see figure \ref{fig:twostage}). Thus, in real world application, the risk measurement requires calculation of a two-level nested Monte Carlo simulation. Lastly, the risk evaluation is done by a risk measure $\rho$, a functional that maps future random exposure to a real number. Examples of risk measure can be \emph{value at risk}, \emph{conditional value at risk}, \emph{probability of loss}, \emph{etc}. Such evaluation structure leads to a challenging computation task. Especially, the inner step simulation has to be done for each scenario generated in the first stage. A lot of research has been done to address the computation issue, to name a few, Gordy and Juneja \cite{Gordy}, Lee and Glynn \cite{Lee}, or Lesnevski \emph{et al.} \cite{Les1, Les2}.

\begin{figure}
\label{fig:twostage}
\center
\vspace{-2ex}
\includegraphics[width = \linewidth]{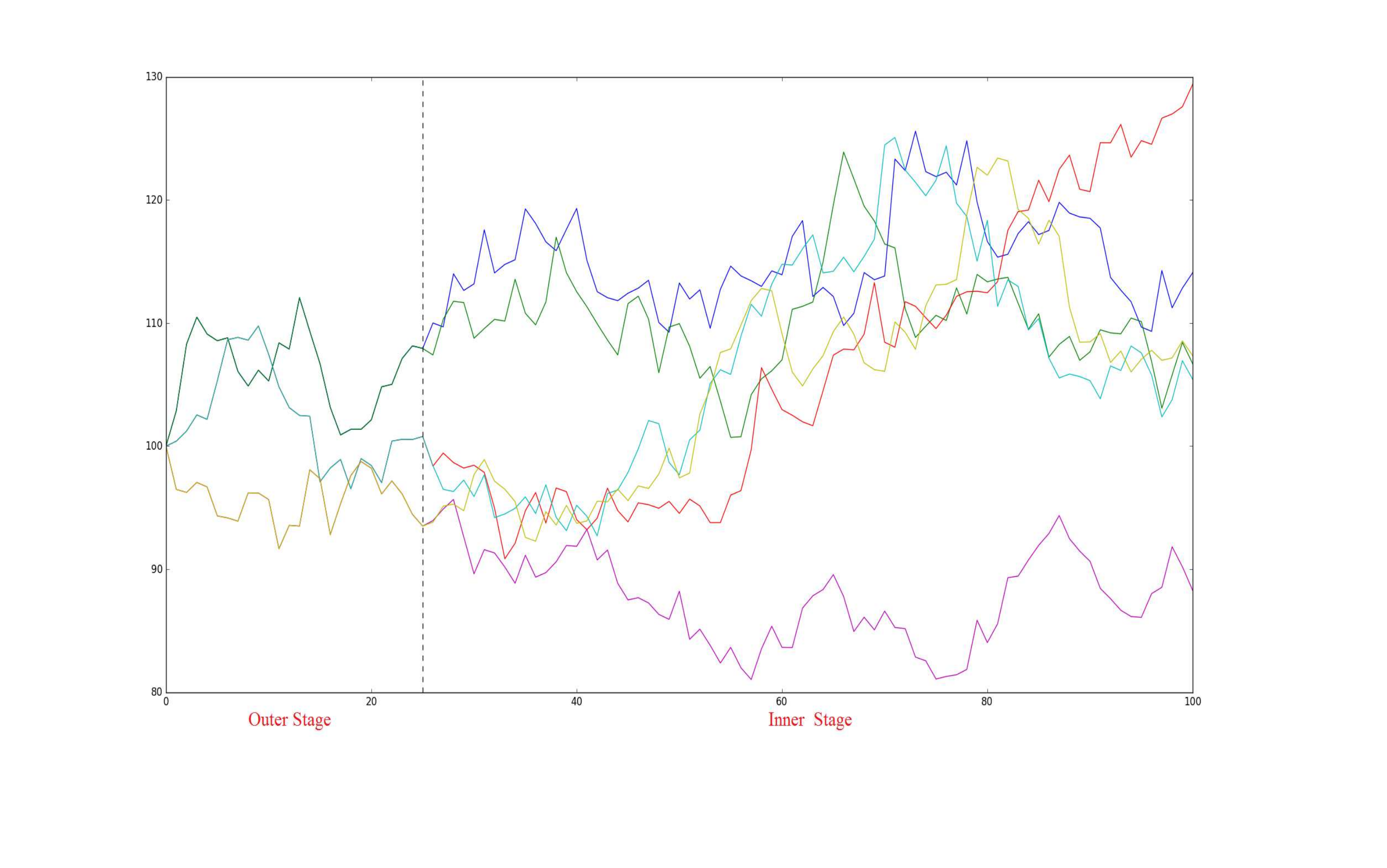} \vspace{-8ex}
\caption{The two stage procedure: the  $x$-axis is time, the $y$-axis is the stock price.}
\end{figure}

The common objective is measuring the risk of a portfolio of assets at the risk horizon $t = \tau$, while standing at time~$0$. We denote the current wealth, \emph{i.e.}, the net present value of portfolio, by  $X_0$ (a known quantity). At time $\tau$, the value of the portfolio is then a $\mathcal{F}_\tau$-measurable random variable $X_\tau$. In almost all real-world applications, we assume a probabilistic model of the evolution of uncertainty between $0$ and $\tau$, for example, a stochastic differential equation. Suppose the outcome $\varOmega$ is a set of possible future scenarios, each of which incorporates sufficient information so as to determine all asset prices at the risk horizon. Then, in each scenario $\omega\in\varOmega$, the portfolio has value $X_\tau(\omega)$. The \emph{mark-to-market (MTM)} loss of this portfolio at time $\tau$ in scenario $\hat{\omega}\in\varOmega$ is given by
\begin{align}
L(\hat{\omega}) = X_0 - X_\tau(\hat{\omega}).
\end{align}
The usual risk measurement is static in the sense that it evaluates the risk of exposure at risk horizon only at current time~$0$. If one wants to check the risk at an intermediate point of the risk horizon, the whole model has to be re-run. Given the computation efforts of nested simulation, it can be very burdensome. In addition, re-simulating can cause inconsistency of risk evaluation, which is also undesired.

In our work, we use dynamic risk measure and the approximation algorithm \eqref{tildeV} proposed in
section~\ref{s:discrete-approximation}, to measure the risk associated with the portfolio dynamically. In this way, the risk can be monitored continuously and consistently; in other words, for any time instant $t$ within the risk horizon, the evolution of risk can be traced.

To better illustrate the dynamic risk evaluation, let us consider a specific example, a portfolio consisting of a long position in a single vanilla put option, which expires at time $T$ and has strike price $K$. The underlying stock, say \emph{ABC}, follows a geometric Brownian motion with an initial price $S_0$, mean $\mu$ and volatility $\sigma $; under the real-world probability measure $\mathbb{P}$, its dynamics is given by the following SDE:
\begin{align}\label{eq:stock}
\frac{dS_t}{S_t} = b\;dt + \sigma\; dW_t, \quad t\in[0, T].
\end{align}
Here, $\{W_t\}$ is $\mathbb{P}$-Brownian motion. Let us also set a flat interest rate $r$; therefore, under the risk-neutral pricing framework, we have the stock dynamics:
\begin{align}\label{eq:riskneutral}
\frac{dS_t}{S_t} = r\;dt + \sigma\; d\widetilde{W}_t,
\end{align}
where $\widetilde{W}_t$ is a $\mathbb{Q}$-Brownian motion. With these specifications, the initial value of the put can be easily calculated by the \emph{Black-Scholes (BS)} formula, which yields
\[
\mathcal{P}(0, S_0)  := \textit{BS}(0, S_0, \sigma, K, T)
 =  S_0N(d_+(T, S_0)) - KD(0, T)N(d_-(T, S_0));
\]
here $N$ stands for the standard normal cumulative distribution function and
\begin{align*}
d_+ (\tau, x) & = \frac{1}{\sigma\sqrt{\tau}}\big[\ln\frac{x}{K} + (r+{\half}\sigma^2)\tau\big], \\
d_- (\tau, x) & = \frac{1}{\sigma\sqrt{\tau}}\big[\ln\frac{x}{K} + (r - {\half}\sigma^2)\tau\big].
\end{align*}
Let us fix a risk horizon $\tau$, denote the price at the risk horizon as $S_\tau(\omega)$. Then, the exposure (MTM) at time $\tau$ is the difference of initial put price $\mathcal{P}(0, S_0)$ and the risk-neutral price of the option at time $\tau$, i.e.,
\begin{align}\label{eq:exposure}
\varPhi(S_\tau(\omega)) := \mathcal{P}(0, S_0) - \Eb^{\mathbb{Q}}\big[\,(S_T - K)^+\,\big|\, S_\tau(\omega)\big].
\end{align}
Here, to get $S_\tau(\omega)$, we have to simulate the path of stock under the real-world measure, \emph{i.e.}, according to \eqref{eq:stock}. Then, to compute the right-hand side, we only need to work out the second term; again, it can be computed analytically by the BS formula. It is well known that the loss function $\varPhi(\cdot)$ is Lipschitz in the state. We are now in the situation to apply a dynamic risk measure,
\begin{align}
\rho_{t, \tau}^g\big[\,\varPhi(S_\tau)\,\big] := Y_t,\quad \text{where} \quad Y_t = \varPhi(S_\tau) + \int_t^\tau g(s, Z_s)\;ds - \int_t^\tau Z_s\;dW_s,\quad t\in[0, \tau],
\end{align}
which enables us to view the risk at any time $t\in[0,\tau]$.

As for the implementation details, instead of using Monte Carlo simulation, we use a tree model for the outer stage. Our risk evaluation algorithm \eqref{tildeVapp} reduces the functional optimization to vector optimization at every time discretization step. However, since backward induction has to be implemented, the state space also needs to be discretized, which makes the tree structure appealing.

\begin{remark} For a diversified portfolio, we shall do a multi-dimensional outer loop simulations, because multiple assets are involved. Also, more sophisticated underlying dynamics is possible, such as  stochastic volatility, local volatility model, in which cases Monte Carlo simulation has to be performed. At the risk horizon, all exposures should be netted before  the $g$-evaluation by our algorithm for BSDEs.
\end{remark}

We now present  the numerical results based on the following data:
$K = 95$, $T = 0.75$, $S_0 = 100$, $\mu = 0.08$, $\sigma = 0.2$, $r = 0.03$, $\tau = 0.2$.
For the risk evaluation, we specify the driver to be:
\[
g(z) = \gamma \|\max\{z \Nc, 0\}\|_p, \quad \Nc\sim N(0, 1),
\]
where the parameters $\gamma>0$ and $p\ge 1$ model risk aversion. The corresponding dual set  is then:
\[
\mathcal{A} = \partial g(0) = \{\,l\in\mathbb{R}^n_+ : |l|_q \le \gamma k\,\},
\]
with $1/p + 1/q=1$, and
\[
k=
\begin{cases}
\frac{1}{\sqrt{2}}(2m(2m-1)\cdots (m+1))^\frac{1}{2m}, & \quad \text{if } p = 2m,\\
(2^m\sqrt{2}\pi m!)^\frac{1}{2m+1},  & \quad \text{if } p = 2m + 1.\\
\end{cases}
\]
Fix $p = 2$, at time $0$, given $\mathcal{F}_\tau$-measurable loss $\varPhi(\cdot)$ in \eqref{eq:exposure}. Table \ref{tab:convergence table} and Figure \ref{fig:convergence} summarize the valuation when varying the step size and risk tolerance $\gamma$. We can observe convergence of the numerical method,  as the step size decreases, uniformly over the whole range of $\gamma$.

\begin{table}[htbp]
  \centering
  \caption{Risk Valuation Convergence Table}
    \begin{tabular}{r|rrrrrr}
    {step size} & \multicolumn{1}{l}{$\gamma = 0.1$} & \multicolumn{1}{l}{$\gamma = 0.3$} & \multicolumn{1}{l}{$\gamma = 0.4$} & \multicolumn{1}{l}{$\gamma = 0.6$} & \multicolumn{1}{l}{$\gamma = 0.8$} & \multicolumn{1}{l}{$\gamma = 1.0$} \\
    \hline
\midrule
    0.4   & 0.00000 & 0.00000 & 0.00000 & 0.00000 & 0.00000 & 0.00000 \\
    0.2   & 0.03407 & 0.23907 & 0.34080 & 0.54207 & 0.73957 & 0.93217 \\
    0.1   & 0.06371 & 0.37300 & 0.52628 & 0.82895 & 1.12492 & 1.41239 \\
    0.08  & 0.07086 & 0.40174 & 0.56573 & 0.88956 & 1.20628 & 1.51403 \\
    0.05  & 0.08261 & 0.44695 & 0.62757 & 0.98446 & 1.33392 & 1.67410 \\
    0.04  & 0.08687 & 0.46282 & 0.64924 & 1.01771 & 1.37878 & 1.73064 \\
    0.02  & 0.09622 & 0.49671 & 0.69544 & 1.08872 & 1.47500 & 1.85268 \\
    0.01  & 0.10165 & 0.51579 & 0.72141 & 1.12877 & 1.52971 & 1.92284 \\
    0.008 & 0.10287 & 0.51998 & 0.72712 & 1.13760 & 1.54184 & 1.93852 \\
    0.005 & 0.10485 & 0.52674 & 0.73632 & 1.15187 & 1.56152 & 1.96407 \\
    0.004 & 0.10557 & 0.52919 & 0.73965 & 1.15705 & 1.56869 & 1.97344 \\
    0.002 & 0.10720 & 0.53465 & 0.74709 & 1.16864 & 1.58483 & 1.99465 \\
    0.001 & 0.10822 & 0.53798 & 0.75163 & 1.17574 & 1.59479 & 2.00786 \\
    0.0008 & 0.10845 & 0.53876 & 0.75269 & 1.17740 & 1.59713 & 2.01099 \\
    0.0005 & 0.10886 & 0.54007 & 0.75447 & 1.18021 & 1.60109 & 2.01630 \\
    0.0004 & 0.10901 & 0.54057 & 0.75515 & 1.18127 & 1.60260 & 2.01833 \\
    \end{tabular}%
  \label{tab:convergence table}%
\end{table}%

\begin{figure}
\center
\vspace{1ex}
\includegraphics[width = 0.8\linewidth]{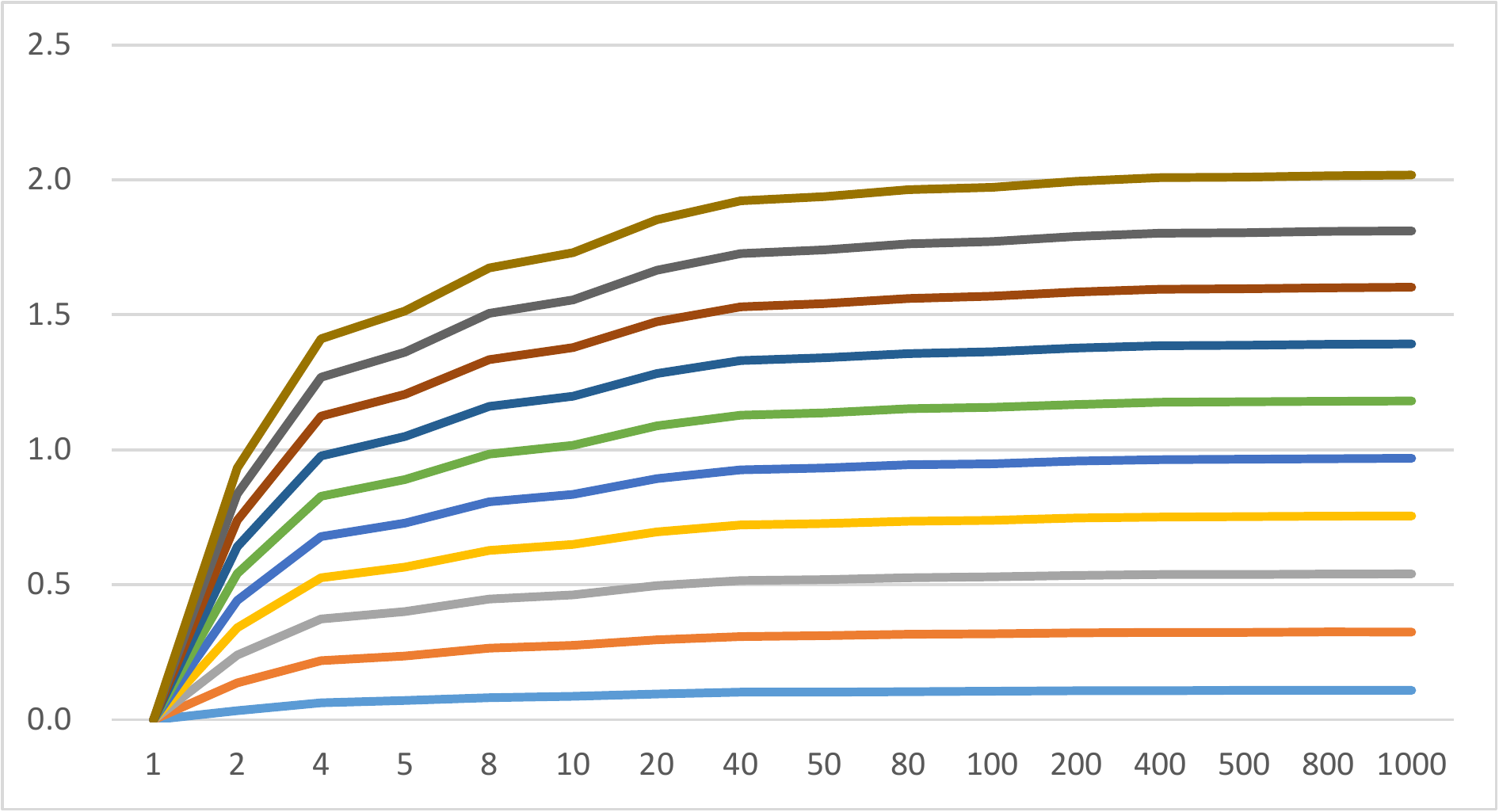}
\caption{Convergence of the discretization method for $\gamma$ ranging from $0.1$ (bottom graph)
to  $1.0$ (top graph); the $x$-axis represents the number of time steps, while the $y$-axis is the risk estimate.}
\label{fig:convergence}
\end{figure}

If we vary the underlying asset's volatility as well as the strike price of the contract, we can construct the \emph{risk surface}. As Table \ref{tab:risksurface} and Figure \ref{fig:risksurface} show, the risk is plotted against different combinations of volatility $\sigma$ and strike price $K$.

\begin{table}[htbp]
  \centering
  \caption{Risk Surface Table}
    \begin{tabular}{rrrrrrrrr}
    \multicolumn{1}{l}{$K$, $\sigma$} & 0.1   & 0.3   & 0.5   & 0.6   & 0.7   & 0.8   & 0.9   & 1 \\
    \midrule
    70    & -0.0002 & -0.1479 & -0.0901 & 0.0566 & 0.2612 & 0.5101 & 0.7918 & 1.0962 \\
    80    & -0.0041 & -0.0188 & 0.2444 & 0.4645 & 0.7290 & 1.0279 & 1.3517 & 1.6918 \\
    90    & 0.0737 & 0.3661 & 0.7508 & 1.0114 & 1.3099 & 1.6379 & 1.9869 & 2.3489 \\
    100   & 0.7941 & 1.0211 & 1.4004 & 1.6691 & 1.9782 & 2.3177 & 2.6780 & 3.0506 \\
    110   & 2.4089 & 1.8858 & 2.1561 & 2.4081 & 2.7100 & 3.0475 & 3.4087 & 3.7833 \\
    120   & 3.9179 & 2.8641 & 2.9802 & 3.2007 & 3.4842 & 3.8108 & 4.1655 & 4.5361 \\
    130   & 4.6752 & 3.8600 & 3.8387 & 4.0233 & 4.2831 & 4.5938 & 4.9376 & 5.3003 \\
    \end{tabular}%
  \label{tab:risksurface}%
\end{table}%

\begin{figure}
\center
\includegraphics[width = 0.8\linewidth]{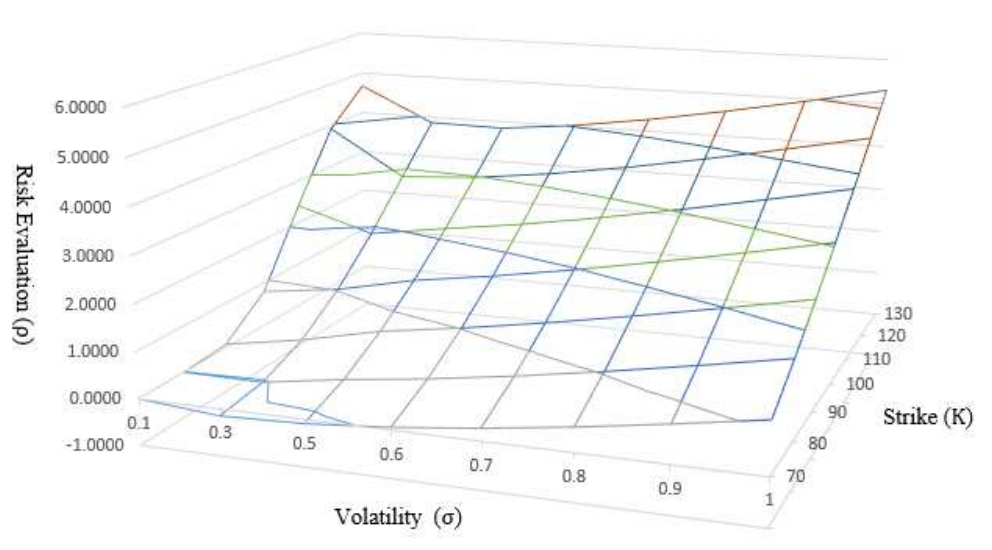}
\caption{The dependence of risk on the strike price and volatility.}
\label{fig:risksurface}
\end{figure}

As we can observe, if the stock \emph{ABC} becomes volatile, the risk of the portfolio should increase, because volatility implies uncertainty. Moreover, since the current stock price is $100$, as the strike price increases, the risk also goes up, which indicates being engaged in an out-of-money trade is riskier than at-the-money or in-the-money. Thus, the risk surface constructed coincides with intuition, which validates the risk evaluation approximation.

\subsection{Risk-Averse Portfolio Pricing}

We consider $n$ stock prices $\big\{S^{(i)}_t\big\}$, $i=1,\dots,n$, following the  system of stochastic differential equations
\begin{align}\label{eq:stockn}
\frac{dS^{(i)}_t}{S^{(i)}_t} = b^{(i)}\;dt + \sigma^{(i)}\; dW_t, \quad t\in[0, T].
\end{align}
Here, $\{W_t\}$ is an $n$-dimensional Brownian motion, $b^{(i)}$ is the drift of stock $i$, and $\sigma^{(i)}$ is the
$n$-dimensional (row) vector of volatility coefficients of stock $i$. For simplicity, we assume that
these coefficients are constant, but nothing in our calculations hinges on that.

At time $T$, a random function
$\varPhi(S_T)$ is evaluated. In our example, it is the value of a portfolio of European call options,
$\varPhi(S_T) = \sum_{i=1}^n x_i \max \big(S_T^{(i)}-K^{(i)},0\big)$, where $x_i$ is the number of option $i$ held, and
$K^{(i)}$ is the strike price. Of course, any other function of $S_T$ is possible here as well.

Let the risk free rate be $r$.
Using the multi-dimensional Girsanov Theorem, we change the real-world probability measure $\mathbb{P}$ to the risk-neutral probability measure $\mathbb{Q}$, and the stock prices follow
the system:
\begin{align}\label{eq:stocknQ}
\frac{dS^{(i)}_t}{S^{(i)}_t} = r\;dt + \sigma^{(i)}\; d\widetilde{W}_t, \quad t\in[0, T].
\end{align}
In this equation, $\widetilde{W}_t$ is an $n$-dimensional Brownian motion with independent components under $\mathbb{Q}$. All calculations are
performed un\-der~$\mathbb{Q}$. The valuation equation is the BSDE:
 \begin{equation}
 \label{int-bsde}
Y_t=\varPhi(S_T)+\int_t^T [-r Y_t +g(Z_\tau)]\;d\tau-\int_t^T Z_\tau \;d\widetilde{W}_r, \quad t \in [0,T].
\end{equation}
If $g \equiv 0$, this reduces to the standard arbitrage-free valuation:
\[
Y_t = e^{-r(T-t)}\Eb^{\mathbb{Q}}\big[ \varPhi(S_T)\big],
\]
but in the risk-averse case the BSDE \eqref{int-bsde} must be solved.

We first simulate the forward process \eqref{eq:stocknQ}.
We apply $n$-dimensional symmetric random walk to approximate the independent Brownian motions. Specifically, we use binomial tree data structure in $n$ dimensions to monitor the stock prices at all lattice points at each discrete time step $t_i=i\Delta_t$, $i=0,1,2,\dots,N$, where  $t_N =T$.

Having the terminal prices for all the stocks, it is easy to generate the terminal states for BSDE part.
To evaluate the dynamic risk in the process, we set the driver $g(z) = \gamma \|z\|$, which satisfies the assumptions for the dual method. Here, $\gamma >0$  is the level of risk aversion. The corresponding dual set (subdifferential) is
$A=\{{a} \;|\;\left\|{a} \right\| \leq \gamma\}$.

In our problem, $\rho^g_0[\varPhi(S_T)]$ is the risk-adjusted price at time 0 for the seller, which means that the seller, given his risk aversion, is willing to sell the option at that price. On the other hand, $-\rho^g_0[-\varPhi(S_T)]$ represents the risk-adjusted price at time 0 for the buyer. The prices are different, of course, because the risk measure is not linear:
$-\rho^g_0[-\varPhi(S_T)] < \rho^g_0[\varPhi(S_T)]$.

In the dual method, at each node of the lattice, the optimization problem \eqref{tildeVapp} was solved by the
gradient method with projection (see, e.g. \cite[Sec. 6.1.1]{ruszczynski2006nonlinear}), which has an easy closed-form representation, due to the fact that $A$ is just a ball.
The normal distribution was represented by the $2^n$ points of the $n$-dimensional random walk.

For comparison, we also used the $\theta$-method of \cite{touzi}, with $\theta = 1/2$. Both methods
were applied to evaluate the buyer's price on an example with $n=d=5$ and
\begin{gather*}
S_0 = K = [50,60,70,80,90],\ r=0.03, \ \gamma = 0.5,\ x= (0.2,0.2,0.2,0.2,0,2),\\
{\vspace{-1ex}}\\
\sigma = \begin{bmatrix}
0.1 & 0.2 & 0.3 & -0.2 & 0.15 \\
0.3 & 0.1 & -0.15 & 0.3 & 0.22\\
0.2 & 0.15 & 0.2 & 0.5 & 0.1 \\
0.1 & 0.3 & 0.2 & -0.1 & 0.05\\
-0.13& 0.3 & 0.1 & 0.1 & -0.15
\end{bmatrix}.
\end{gather*}
The results are depicted in Figure \ref{f:portfolio5}. The horizontal axis is the number of discretization points $N$.

\begin{figure}[htbp]
\vspace{-5ex}
\centering
\includegraphics[width=0.7\textwidth]{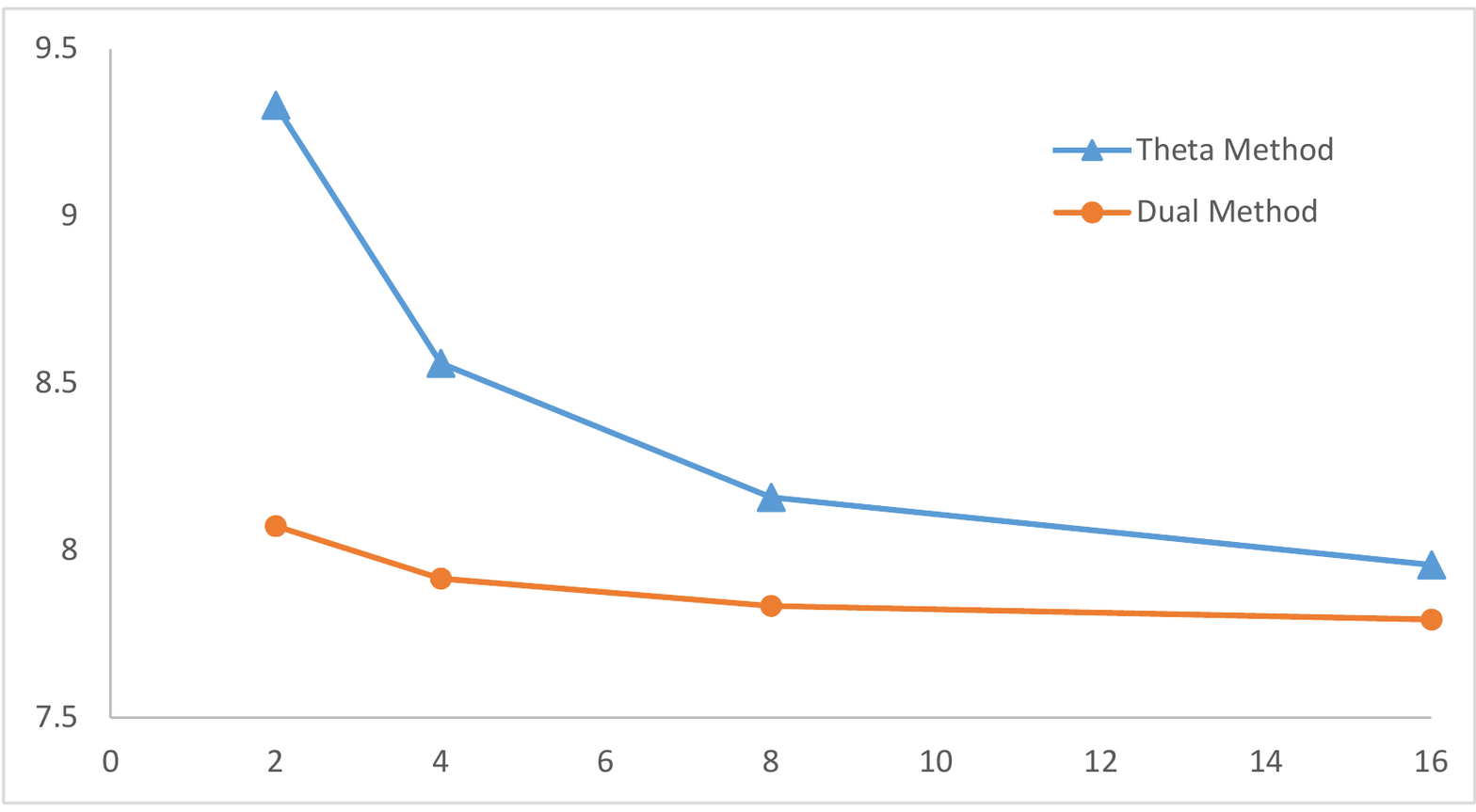}
\vspace{-5ex}
\caption{Convergence of the Dual Method (lower graph) and the $\theta$-Method (upper graph) for the buyer, when $\gamma$ = 0.5.}.
\label{f:portfolio5}
\end{figure}

We can see from the figure that the dual method exhibits much faster convergence than the $\theta$-method.
This is due to the fact that it exploits the convexity of the driver in the dual optimization step \eqref{tildeVapp},
thus minimizing the numerical error accumulated.

The results have been borrowed from extensive tests carried out in \cite{Hu2018} for the buyer's and seller's prices
for several examples with different dimensions and different risk aversion coefficient $\gamma$.
In all experiments reported there, the pattern is very similar
to Figure \ref{f:portfolio5}: the dual method significantly outperforms the $\theta$-method. We believe
that for convex drivers, the dual method is a very attractive numerical approach to backward stochastic differential equations.

\section*{Acknowledgments}
The authors thank Yuanhan Hu and Bo Ni for the permission to use the results of the second example.
The authors acknowledge the Office of Advanced Research Computing (http://oarc.rutgers.edu) at Rutgers, The State University of New Jersey, for providing access to the Amarel cluster and associated research computing resources that have contributed to the results reported here.

\end{document}